\documentclass[12pt]{article}

\usepackage{amsmath,amsthm,amssymb}
\usepackage{hyperref}

\usepackage{graphicx, enumerate, url}
\usepackage{color}
\usepackage{tikz}

\numberwithin{equation}{section}

\theoremstyle{plain}
\newtheorem{theorem}[equation]{Theorem}
\newtheorem{lemma}[equation]{Lemma}

\newtheorem{proposition}[equation]{Proposition}

\theoremstyle{definition}
\newtheorem{definition}[equation]{Definition}

\theoremstyle{remark}
\newtheorem{remark}[equation]{Remark}

\begin{document}

\title{Convex hulls of polynomial Julia sets}         
 \author{Ma{\l}gorzata~Stawiska\\Mathematical Reviews\\ 416 Fourth Street\\ Ann Arbor, MI 48103-4816, USA\\stawiska@umich.edu}

\maketitle


\begin{abstract} We prove P. Alexandersson's conjecture that for every complex polynomial $p$ of degree $d \geq 2$ the convex hull $H_p$ of the Julia set $J_p$ of $p$ satisfies $p^{-1}(H_p) \subset H_p$. We further prove that   the equality $p^{-1}(H_p) = H_p$ is achieved only if $p$ is affinely conjugated to the Chebyshev polynomial $T_d$ of degree $d$, to $-T_d$ or a monomial $c z^d$ with $|c|=1$.
\end{abstract}

\section{Introduction} 

Let $d$ be a positive integer greater than or equal to $2$ and let $a_0,a_1,...,a_d$ be complex numbers such that $a_d \neq 0$. Then $p(z)=a_dz^d+...+a_1z+a_0, z \in \mathbb{C}$  is a polynomial of degree $d$. In particular,

 \[
 \lim_{|z| \to \infty} \frac{|p(z)|}{|z|^d} >0.
 \]
 
  This  property  influences dynamics of $p$ in the complex plane. It implies that there exists an $R >0$ such that $p^{-1}(D_R) \subset D_R$, where $D_R:=\{z: |z|\leq R\}$.  Furthermore (cf. \cite{Kl}, Corollary 6), for any such  $R$ and for each positive integer $k_0$ we have 
  \[
 \emptyset \neq  K_p= \bigcap_{k \geq k_0} p^{-k}(D_R),
\]

where $K_p:=\{z \in \mathbb{C}: \{p^{\circ n}(z)\}\  \text{is bounded}\}$. We use the notation $p^{\circ n}$ to denote  the composition $p \circ ... \circ p$, $n$ times.\\

Considering $p$ as a holomorphic map of the Riemann sphere $\widehat{\mathbb{C}}$ onto itself (with $p(\infty)=\infty)$, we  define the Fatou set $F_p$ of $p$ and the Julia set $J_p$ of $p$ as follows (cf. Definition 3.1.3, \cite{Be}): $F_p$ is the maximal open subset of  $\widehat{\mathbb{C}}$ on which the sequence $\{p^{\circ n}, n \in \mathbb{N}\}$ is equicontinuous, and $J_p$ is the complement of $F_p$ in $\widehat{\mathbb{C}}$.  By Theorem 4.2.1 in \cite{Be} $J_p$ is infinite. For a general holomorphic (hence rational) map $r$ of  $\widehat{\mathbb{C}}$ Theorem 4.2.3 in \cite{Be} gives an alternative that either $\mathcal{J}_r=\widehat{\mathbb{C}}$  or $\mathcal{J}_r$ has empty interior. However, by Theorem 3.2.5 in \cite{Be}, for any polynomial  $p$ the fixed point $\infty$ belongs to the Fatou set, so $J_p$ always has empty interior. By $F_\infty$ we denote the connected component of  $F_p$ containing $\infty$.\\

It follows from the above definitions that $K_p$ is the complement of $F_\infty$ in $\mathbb{C}$ or, equivalently, the union of $J_p$ with bounded components of $F_p$. We call $K_p$ the filled-in Julia set of $p$. Using Montel's theorem it is easy to deduce from the above definitions that $J_p=\partial F_\infty =\partial K_p$.\\

By Theorem 3.2.4 in \cite{Be}, the Fatou set and the Julia set are completely invariant under $p$. That is, $p^{-1}(F_p)=F_p=p(F_p)$ and similarly $p^{-1}(J_p)=J_p=p(J_p)$. The above representation of $K_p$ shows that  $p^{-1}(K_p)=K_p=p(K_p)$. In view of these equalities, as well as the inclusion $p^{-1}(D_R) \subset D_R$, on which they rely, it is natural to ask about existence of   closed sets $E \subset \mathbb{C}$ (other than $J_p$, $K_p$ or $D_R$) containing $J_p$ such that $p^{-1}(E) \subset E$. In \cite{Al} P.~Alexandersson conjectured that $p^{-1}(H_p) \subset H_p$ for every complex polynomial $p$ of degree $d \geq 2$, where $H_p={\rm conv}J_p$ is the convex hull of the Julia set of $p$. He presented graphical evidence in support of this conjecture in some cases of quadratic and cubic polynomials. We settled his conjecture positively; see \cite{St} for a preliminary version.\\

In this article we present a detailed proof of  Alexandersson's conjecture that $p^{-1}(H_p) \subset H_p$ for every complex polynomial $p$ of degree $d \geq 2$ (Theorem \ref{theorem: main}). In addition we characterize the polynomials $p$ for which  the equality $p^{-1}(H_p) = H_p$ is achieved (Theorem \ref{theorem: conjcheb} and Theorem \ref{theorem: conjpower}).\\

Let us first mention some simple  examples of polynomials $p$ with $H_p=p^{-1}(H_p)$.  For $p(z)=T_2(z)=2z^2-1$, the Julia set $J_p$ is the real axis segment $[-1,1]$, so  $H_p=J_p$ and of course $H_p=p^{-1}(H_p)$. The same holds  for any $d > 2$ and the Chebyshev polynomial $T_d$ of degree $d$,  $T_d(\cos z)=\cos (dz)$. It is also possible to give an example with $J_p \subsetneq H_p$: for any $d \geq 2$, $|\alpha|=1$ and  $p(z)=\alpha z^d$, we have $J_p=\{|z|=1\}$, $H_p=\{|z|\leq 1\}$ and $H_p=p^{-1}(H_p)$. We will show that, up to a conjugacy by an affine map $g(z)=az+b$, these are the only examples achieving the equality $p^{-1}(H_p) = H_p$. \\

We will use the following:

\begin{proposition} (Theorem 1.4.1, \cite{Be})  \label{prop: cheb} Let $p$ be a polynomial of degree $d \geq2$. Then the real interval $[-1,1]$ is both forward and backward invariant under $p$ if and only if $p$ is $T_d$ or $-T_d$, where $T_d$ is the Chebyshev polynomial of degree $d$.
\end{proposition}

\begin{proposition}  (Theorem 1.3.1, \cite{Be}) \label{prop: powers} Let $p$ be a polynomial of degree $d \geq2$ and suppose that the unit circle $\{|z|=1\}$ is both forward and backward invariant under $p$. Then $p(z)=c  z^d$, where $|c|=1$. 
\end{proposition}

\begin{proposition} (\cite{Be}, Theorem 3.1.4): \label{prop: moebius} Let $r$ be a non-constant rational map, let $g$ be a M\"obius map,  $g(z)=\frac{\alpha z+\beta}{\gamma z+ \delta}$, and let $s=g\circ r \circ g^{-1}$. Then $\mathcal{F}_s=g(\mathcal{F}_r)$ and  $\mathcal{J}_s=g(\mathcal{J}_r)$. 
\end{proposition}

 
  %
      
  \section{A primer on convex sets} 
  
  Let us recall some basic definitions and theorems about convexity in vector spaces over $\mathbb{R}$. We will use these in $\mathbb{C}=\mathbb{R}+i\mathbb{R}$.
  
\begin{definition}  (\cite{Ho2}, Definition 2.1.1; \cite{Ba}, p. 5): Let $V$ be a vector space over $\mathbb{R}$. A subset $X$ of $V$ is
called convex if $tx_1 + (1-t)x_2 \in X$ whenever $0\leq t \leq 1$ and $x_1,x_2 \in X$.
\end{definition} 
  
    \begin{definition} (cf. Proposition 2.1.3, \cite{Ho2} or \cite{Ba}, Theorem (I.2.1)) Let $V$ be a vector space over $\mathbb{R}$.  For every subset $X \subset V$ the intersection ${\rm conv} X$ 
of all convex sets containing $X$ is  a convex set, called the convex
hull of $X$.  
 \end{definition}  
 
 It follows that for any two sets $X_1$, $X_2$ in $V$, if  $X_1 \subset X_2$, then ${\rm conv} X_1 \subset {\rm conv} X_2$. Furthermore, a set is convex if and only if it equals its convex hull.\\ 
 
 Convexity is invariant under affine maps:
 
\begin{proposition} (\cite{Ho2}, Proposition 2.1.2): If $T$ is an affine map $V_1 \to  V_2$ where $V_j$ are vector
spaces, and $X_j$ is a convex subset of $V_j$, then $TX_1 = \{Tx: x \in X_1\}$ and
$T^{-1}X_2 = \{x \in V_1: Tx \in X_2\}$ are convex subsets of $V_2$ and $V_1$, respectively.
\end{proposition} 

 We will use the following version of the hyperplane separation theorem:
 
\begin{proposition} ( \cite{Ho2},  Corollary 2.1.11; \cite{Ba}, Theorem (III.1.3)) \label{lemma: separation}: Let $X$ be a convex and closed subset of a finite-dimensional
vector space $V$. If $x_0 \not \in X$, then  there is an affine half-space containing $x_0$ which does not intersect $X$; that is, there is an affine  function $f: V \to \mathbb{R}$ with $f(x_0) < 0 \leq f(x), \ x \in X$. 
  \end{proposition}   
  
 Some other useful properties  of convex sets in $\mathbb{R}^n$   are the following:    
                
\begin{proposition} (\cite{Ba}, Theorem (III.2.5): If $X \subset \mathbb{R}^n$ is convex, then the interior of $X$ is also convex.
 \end{proposition}
 
 \begin{proposition}
(\cite{Ho2}, Theorem 2.15; \cite{Ba}, Corollary (I.2.4)): If $K \subset \mathbb{R}^n$ is compact, then ${\rm conv } K$ is also compact.
 \end{proposition}
 
 The following topological property of convex sets will also play an important role.
 
 \begin{proposition} \label{prop: boundary}The boundary of a convex set $X$ with nonempty interior in $\mathbb{R}^n$ is homeomorphic to the unit sphere.
 \end{proposition}
 
 \begin{proof} It can be assumed that the origin $0$ is in the interior of $X$. The map $\varphi: \partial X \ni x \mapsto x/\|x\| \in \mathbb{S}^{n-1}$ is a homeomorphism.
 For further details, see Theorem 2.1, \cite{nlab}.
 \end{proof}
 
 A relation between convex sets and behavior of complex polynomials in $\mathbb{C}$ is given by the classical Gauss-Lucas theorem:

  \begin{theorem} (\cite{RS}, Theorem 2.1.1; \cite{CGOT}, Theorem 1.1; \cite{Ba}, problem 5, p.8; \cite{Ho2}, Exercise 2.1.14): Every convex set in the complex plane containing all the zeros of a complex polynomial $p$ also contains all critical points of $p$.
  \end{theorem}   
      
We will also use the following version of Gauss-Lucas theorem, due to W.~P.~Thurston:

\begin{lemma} (\cite{CGOT}, Proposition 1.1) \label{lemma: Thurston}: Let $p$ be any polynomial of degree at least two. Denote by $\mathcal{C}$ the convex hull of the critical points of $p$. Then $p: E \to \mathbb{C}$ is surjective for any closed half-plane $E$ intersecting $\mathcal{C}$.
\end{lemma} 


The following consequence of Gauss-Lucas theorem is a modification of Exercise 2.1.15 in \cite{Ho2}:

\begin{lemma}\label{lemma: convexity}
Let  $p(z)=\sum_{j=0}^d a_jz^j$ be a polynomial in $z \in \mathbb{C}$ of degree $d$.  Let $B$ be a closed convex subset of $\mathbb{C}$ containing zeros of $p'$. Then the set $C_B$ of all $w \in \mathbb{C}$ such that all the zeros of $p(\cdot)-w$ are contained in $B$ is a convex set.
\end{lemma}

\begin{proof} Note that by continuity of roots (see e.g. \cite{Lo}, Section B.5.3)  $C_B$ is closed when $B$ is. Let $w_1,w_2 \in C_B$ and $n_1,n_2 \in \mathbb{N}$ and consider the polynomial (in one complex variable $z$) $P(z):=(p(z)-w_1)^{n_1}(p(z)-w_2)^{n_2}$. Then all zeros of $P$ lie in $B$ (by definition of $C_B$), so the convex hull of zeros of $P$ is contained in $B$. By Gauss-Lucas theorem, all zeros of $P'$ are contained in $B$. The zeros of $P'$ are respectively all the zeros of $p(z)-w_1$, all the zeros of  $p(z)-w_2$ (if $n_1,n_2 >1$), all the zeros  of $p'$ and all the zeros of $p(\cdot)-\left (\frac{n_2}{n_1+n_2}w_1 + \frac{n_1}{n_1+n_2}w_2\right)$. By the definition of $C_B$,  $\frac{n_2}{n_1+n_2}w_1 + \frac{n_1}{n_1+n_2}w_2 \in C_B$. Varying $n_1,n_2$ and using the property that $C_B$ is closed, we get that $tw_1+(1-t)w_2 \in B$ for all $0 \leq t \leq 1$.
\end{proof}

The following simple observation will be also useful:
 
 \begin{lemma} \label{lemma: triangle} Let $K \subset \mathbb{R}^2$ be a convex set containing more than one point. Then $K$ has an empty interior if and only if it is a subset of a line $\{(x,y): ax+by+c=0\}$, where $a^2+b^2\neq 0$.
 \end{lemma}
 
 \begin{proof} The ``if'' part is obvious. For the ``only if'', assume that $K$ contains three non-collinear points $z_1, z_2, z_3$. Then it also contains the convex hull of $\{z_1, z_2, z_3\}$. This convex hull is the triangle with vertices $z_1, z_2, z_3$, which has non-empty interior.
 \end{proof} 
 
In addition to the above  notion of a convex hull, we will use the notion of a holomorphically convex hull of a compact subset of $\mathbb{C}$ and some of its properties.

  \begin{definition} (\cite{Ho1}, p. 8) Let $\Omega$ be an open set in $\mathbb{C}$. Let $A(\Omega)$ denote the class of holomorphic functions  in $\Omega$. Let $Z$ be an arbitrary compact subset of $\Omega$. Then the holomorphically convex hull $\widehat{Z}$ of $Z$ in $\Omega$ is defined as
  \[ 
  \widehat{Z}=\widehat{Z}_\Omega=\{z \in \Omega: |f(z)| \leq \sup_Z |f| \ \forall f \in A(\Omega)\}.
  \]
  \end{definition}
  
\begin{proposition} \label{prop: hulls} (\cite{Ho1}, p. 8) For $Z$ as above we  have
\[
\widehat{Z} \subset {\rm conv} Z. 
\]
\end{proposition}

\begin{proposition} \label{prop: Runge} (\cite{Ho1}, Theorem 1.3.3) $\widehat{Z}_\Omega$ 
is 
the 
union 
of 
$Z$ 
and 
the  connected components 
of 
$\Omega \setminus 
Z$ 
which 
are 
relatively  compact 
in 
$\Omega$. 
\end{proposition}

\section{Proof of the Theorem}

Our main result is the following:

\begin{theorem} \label{theorem: main} Let $p$ be a complex polynomial of degree $d \geq 2$. Then $p^{-1}(H_p) \subset H_p$.
 \end{theorem}
 
 We will use the Lemma below, a consequence of Lemma \ref{lemma: Thurston}:

\begin{lemma} \label{lemma: allcritpoints} Let $p$ be any polynomial of degree at least two. Then all zeros of $p'$ belong to $H_p={\rm conv}J_p$, the convex hull of the Julia set $J_p$ of $p$. 
\end{lemma}

\begin{proof} Suppose there is an $x_0 \not \in H_p$ such that $p'(x_0)=0$. By Proposition \ref{lemma: separation} (applied twice if necessary), there exists a closed half-plane $E$ such that $x_0 \in E$ and $E \cap J_p = \emptyset$. By Lemma \ref{lemma: Thurston}, $p: E \to \mathbb{C}$ is surjective. Take a $z_0 \in J_p$. Then on one hand $p^{-1}(z_0) \subset J_p$, while on the other hand $p^{-1}(z_0) \cap E \neq \emptyset$, a contradiction. 
\end{proof}

\begin{remark} For the quadratic family $p_c(z)=z^2+c, \ c \in \mathbb{C}$, which gave motivation for \cite{Al},    it is easy to check directly (without appealing to Lemma \ref{lemma: allcritpoints}) that the critical point $0$ is the center of symmetry of the Julia set $J_c$, so it is a convex combination of two points in $J_c$.
\end{remark}

 \begin{proof} (of Theorem \ref{theorem: main}) 
  By Lemma \ref{lemma: allcritpoints} $B=H_p$ satisfies the assumptions of Lemma \ref{lemma: convexity}. Applying  Lemma \ref{lemma: convexity} to   $H_p$  we get that the set $C_p=\{ w \in \mathbb{C}: p^{-1}(w) \in H_p\}$ is convex. Furthermore, for $w \in J_p$ we have $p^{-1}(w) \in J_p \subset H_p$, so $J_p \subset C_p$. Hence $H_p \subset C_p$, which implies $p^{-1}(H_p) \subset H_p$. 
\end{proof}

To characterize polynomials $p$ with $H_p=p^{-1}(H_p)$, we first prove the following two lemmas:



\begin{lemma} For every complex polynomial $p$ of degree $d \geq 2$, $K_p \subset H_p$. 
\end{lemma}

\begin{proof}
Note that by Proposition \ref{prop: Runge} $K_p$ is the holomorphically convex hull of $J_p$ in $\Omega=\mathbb{C}$. By Proposition \ref{prop: hulls}, the holomorphically convex hull  of a compact set in $\mathbb{C}$ is a subset of the  convex hull of this compact, which concludes the proof.
\end{proof}

\begin{lemma} Let $p$ be a complex polynomial of degree $d \geq 2$. If $H_p=p^{-1}(H_p)$, then $H_p \subset  K_p$.
\end{lemma}

\begin{proof} If $H_p=p^{-1}(H_p)$, then $H_p = \bigcap_{n \geq 0} p^{-n}(H_p)$. It follows that for every $z \in H_p$ the orbit $\{p^{\circ n}(z): n \in \mathbb{N}\}$ is contained in $H_p$, and hence bounded. Therefore $H_p \subset  K_p$.
\end{proof}

Now we consider two cases: $J_p = K_p$ or $J_p \subsetneq K_p$.

\begin{theorem} \label{theorem: conjcheb}  Let $p$ be a complex polynomial of degree $d \geq 2$ such that $H_p=p^{-1}(H_p)=J_p$. Then $p$ is affinely conjugated to the Chebyshev polynomial $T_d$ or to $-T_d$.
\end{theorem}

\begin{proof} Recall that for any polynomial $p$ the Julia set $J_p$ has empty interior. If $J_p=H_p$, then $J_p$ is an infinite closed convex set in $\mathbb{C}$ with empty interior, and hence, by Lemma \ref{lemma: triangle}, is a subset of a line.  Being connected and compact, $J_p$ must be a (closed) segment. An affine map $g(z)=az+b$  transforms $J_p$ into $[-1,1]$, which is the Julia set of $g\circ p \circ g^{-1}$. By Proposition \ref{prop: cheb}, $g\circ p \circ g^{-1}$ is equal to the Chebyshev polynomial $T_d$ or to $-T_d$.
\end{proof}

\begin{theorem} \label{theorem: conjpower}  Let $p$ be a complex polynomial of degree $d \geq 2$ such that $H_p=p^{-1}(H_p)=K_p$ has nonempty interior. Then $p$ is affinely conjugated to a monomial  $c z^d$ with $|c|=1$.
\end{theorem}

The proof of Theorem \ref{theorem: conjpower} relies on properties of  Hausdorff measure and Hausdorff dimension (henceforth denoted by ${\rm dim}$). We will not define these notions here. For detailed account we refer to \cite{Fa}. Applications to Julia sets can be also found in Chapter 10 in \cite{Be}. The following result is key:

\begin{proposition} \label{prop: rect} (Theorem 1, \cite{Ha}): Let $f :\widehat{\mathbb{C}} \to \widehat{\mathbb{C}}$ be a rational function. Suppose
that the Julia set $J_f$ is a Jordan curve. Then ${\rm dim} (J_f) > 1$ or $J_f$ is a circle/line. 
\end{proposition}

\begin{proof} (of Theorem \ref{theorem: conjpower}) By Proposition \ref{prop: boundary}, $J_p=\partial K_p = \partial H_p$ is a (closed) Jordan curve and the (only) connected boundary component of the convex domain ${\rm int} \ K_p$. 
By Problem 1.5.1 in \cite{To}, every such curve is  rectifiable, that is, it has positive and finite length. By Lemma 3.2 in \cite{Fa}, the Hausdorff dimension of  a rectifiable curve equals $1$ (this consequence is explicitly stated on p. 30 in \cite{Fa}, after the proof of the Lemma 3.2). Because $\infty$ is in the Fatou set of $p$, Proposition \ref{prop: rect} implies that $J_p$ is a Euclidean circle. Transforming $J_p$ by an affine map $g$ into the unit circle and applying Proposition \ref{prop: powers} we conclude that $p(z)=g^{-1}(c (g(z))^d)$, where $|c|=1$. 
\end{proof}

\end{document}